\documentclass{amsart}
\usepackage{amsmath}
\usepackage{amsthm}
\usepackage{amssymb}
\usepackage{epsfig}

\theoremstyle{plain}
\newtheorem{thm}{Theorem}[section]
\newtheorem*{unnum}{Theorem}
\newtheorem*{unnumcor}{Corollary}
\newtheorem{cor}[thm]{Corollary}
\newtheorem{lem}[thm]{Lemma}
\newtheorem{prop}[thm]{Proposition}
\theoremstyle{definition}
\newtheorem{defn}[thm]{Definition}
\newtheorem{example}[thm]{Example}
\theoremstyle{remark}
\newtheorem{rem}[thm]{Remark}

\title{Obtainable Sizes of Topologies on Finite Sets}

\author{K\'{a}ri Ragnarsson}
\address{Mathematics Institute, Reykjav\'{i}k University, Kringlunni 1, 103 Reykjav\'{i}k, Iceland}
\email{kari.ragnarsson@ru.is}

\author{Bridget Eileen Tenner}
\address{Department of Mathematical Sciences, DePaul University, 2320 North Kenmore Avenue, Chicago, IL 60614, USA}
\email{bridget@math.depaul.edu}


\begin{document}

\begin{abstract}
We study the smallest possible number of points in a topological space having $k$ open sets. Equivalently, this is the smallest possible number of elements in a poset having $k$ order ideals.  Using efficient algorithms for constructing a topology with a prescribed size, we show that this number has a logarithmic upper bound.  We deduce that there exists a topology on $n$ points having $k$ open sets, for all $k$ in an interval which is exponentially large in $n$.  The construction algorithms can be modified to produce topologies where the smallest neighborhood of each point has a minimal size, and we give a range of obtainable sizes for such topologies.
\end{abstract}

\maketitle

\section{Introduction}

Finite topological spaces present many interesting combinatorial questions.  The most fundamental of these concerns the number $T(n)$ of different topologies on $n$ points.  This number has been determined by exhaustive enumeration for $n \leq 16$ (\cite{BrinkmannMcKay}).  The general question is very difficult, and it is uncertain whether a formula for $T(n)$ will ever be obtained, although asymptotic estimates exist.  Ern\'e showed in \cite{erne0} and \cite{erne1} that $T(n)$ is asymptotically equal to $T_0(n)$, the number of $T_0$-topologies (or, equivalently, partial orders) on $n$ points, which together with the asymptotic bounds for the latter, due to Kleitman and Rothschild in \cite{kleitman1} and \cite{kleitman2} provide asymptotic bounds for $T(n)$.  Moreover, Ern\'e gave the asymptotic estimate $2^{n/2 + O(\log_2n)}$ for the average cardinality of topologies on $n$ points in \cite{erne2}.

The enumeration of topologies on $n$ points can be refined by counting $T(n,k)$, the number of topologies on $n$ points having $k$ open sets.  Just as for $T(n)$, this is a long-standing open problem, although some special cases are known.  The most important contributions are due to Ern\'e and Stege, who in \cite{ernestege-tech} computed the values of $T(n,k)$, for $n \le 11$ and arbitrary $k$, as well as the related numbers of $T_0$ and connected topologies, and the corresponding numbers of homeomorphism classes. Their results in particular yield all numbers $T(n,k)$ for $k \le 12$, which were later calculated independently by Benoumhani \cite{benoumhani}.  Moreover, Ern\'e and Stege computed the numbers $T(n,k)$ for $k \le 23$ in \cite{ernestege2}.

When $k$ is large in relation to $n$, then certainly $T(n,k) = 0$ for $k>2^n$.  In fact $T(n,k) = 0$ for many large values of $k \le 2^n$.  As a first step in this direction, Sharp \cite{sharp} and Stephen \cite{stephen} showed that $T(n,k) = 0$ when $3 \cdot 2^{n-2} < k < 2^n$.  Stanley \cite{stanley} computed $T(n,k)$ for $k \ge 7 \cdot 2^{n-4}$, and Kolli \cite{kolli} did likewise for $k \ge 3 \cdot 2^{n-3}$.  Additional cardinalities for large $k$ were computed by Parchmann (\cite{parchmann1} and \cite{parchmann2}), and Vollert characterized when $T(n,k) > 0$ for $k \in [2^{n-2},2^n]$ (see \cite{vollert}).

For a given $n$, it is then natural to ask: what is the smallest value of $k$ so that $T(n,k) = 0$?

\begin{defn}
For an integer $n \ge 1$, let $f(n) \ge 2$ be the smallest integer so that there exists no topology on $n$ points having $f(n)$ open sets.
\end{defn}

Equivalently, $f(n)$ is the largest number  so that there exists a topology on $n$ points with $k$ open sets for all $2 \leq k < f(n)$.  It was known as early as the 1970s that $f(n) < 2^{n-2}$ for $n > 8$ (see Parchmann \cite{parchmann1} and \cite{parchmann2}).  Particular examples supporting this result are given below.  These examples already occur in Vollert's thesis \cite{vollert}, and much more comprehensive material can be found in the papers by Ern\'e and Stege (see \cite{ernestege1, ernestege-tech, ernestege2}).  For example, \cite{ernestege-tech} yields $f(n)$ for all $n \le 11$.  From the asymptotic bounds established by Erd\"os and Ern\'e for clique numbers of graphs (see \cite{erdos}), in particular, for antichain numbers of posets, it follows that the quotient $f(n)/2^n$ tends to $0$ when $n$ becomes large.  On the other hand, in \cite{vollert}, Vollert derived the lower bound $f(n) \ge 2^{n/2+1}$ using arguments similar to those in Corollaries \ref{cor:x2+1} and \ref{cor:singles} below.

\begin{example}
There is no topology on $9$ points having $127$ open sets. That is, $T(9,127) = 0$.
\end{example}

\begin{example}
There is no topology on $10$ points having $191$ open sets.  That is, $T(10,191) = 0$.
\end{example}

We reproduced these results by letting Stembridge's MAPLE package \cite{posets} count the order ideals in all isomorphism classes of posets with at most $10$ elements. The relationship between posets and topologies is discussed in Section \ref{sec:Machinery}.

In this paper we obtain exponential lower bounds for $f(n)$, and thus a large interval of integers $k$ for which $T(n,k) > 0$. To this end we introduce and examine the following sequence.

\begin{defn}
For an integer $k \ge 2$, let $m(k)$ be the smallest positive integer such that there exists a topology on $m(k)$ points having $k$ open sets. 
\end{defn}

The above examples can be reformulated as: $m(127) > 9$ and $m(191) > 10$.

In Section \ref{sec:bounds} we obtain logarithmic upper bounds for $m(k)$, the main result being the following.

\begin{unnum}
For all $k \ge 2$,
\begin{equation*}
m(k) \le (4/3) \lfloor \log_2 k \rfloor + 2.
\end{equation*}
\end{unnum}

The proof is constructive. That is, we provide an algorithm to construct a topology with $k$ open sets using no more than $(4/3) \lfloor \log_2 k \rfloor + 2$ points. As $f(n)$ is the smallest value of $k$ such that $m(k) > n$ (cf.~Remark~\ref{rem:explainT0}), the theorem yields the following bound for $f(n)$.

\begin{unnumcor} For all $n \ge 1$,
\begin{equation*}
f(n) > 2^{3(n-2)/4}.
\end{equation*}
\noindent That is, $T(n,k) > 0$ for all $k \in [2,2^{3(n-2)/4}]$.
\end{unnumcor}

Thus this paper focuses on the values $\{m(k)\}$, and finding a close upper bound for the sequence.  The MAPLE program \cite{posets} can compute the initial values of this sequence, presented in Table~\ref{table:min n for k} for $k \in [2,35]$.  This is sequence A137813 of \cite{oeis}.

\begin{table}[htbp]
\begin{center}
$\begin{array}{c|cccccccccccccccccc}
k & 2 & 3 & 4 & 5 & 6 & 7 & 8 & 9 & 10 & 11 & 12 & 13 & 14 & 15 & 16 & 17 & 18 & 19\\
\hline
m(k) & 1 & 2 & 2 & 3 & 3 & 4 & 3 & 4 & 4 & 5 & 4 & 5 & 5 & 5 & 4 & 5 & 5 & 6\\
\end{array}$
\ \\
\vspace{.2in}
$\begin{array}{c|cccccccccccccccc}
k & 20 & 21 & 22 & 23 & 24 & 25 & 26 & 27 & 28 & 29 & 30 & 31 & 32 & 33 & 34 & 35\\
\hline
m(k) & 5 & 6 & 6 & 6 & 5 & 6 & 6 & 6 & 6 & 7 & 6 & 7 & 5 & 6 & 6 & 7
\end{array}$
\end{center}
\smallskip
\caption{The minimum number of points $m(k)$ needed to make a topology having $k$ open sets, as computed by \cite{posets}, for $k \in [2,35]$.}\label{table:min n for k}
\end{table}

The numerical tables computed by Ern\'e and Stege in \cite{ernestege-tech} give $m(k)$ at least for $k \le 379$, and $f(11) = 379$.

The same computation also gives us the values of $f(n)$ for $n \in [1,10]$. These values are displayed in Table~\ref{table:min bad k for n}, where they are also compared to the result of Theorem~\ref{thm:triples}. The table indicates, as expected, that the bound is not strict.  However, these data points do not contradict the possibility that $2^{3(n-2)/4}$ may give the correct growth rate for $f(n)$.

\begin{table}[htbp]
\begin{center}
$\begin{array}{c|cccccccccc}
\rule[-2mm]{0mm}{6mm} n & 1 & 2 & 3 & 4 & 5 & 6 & 7 & 8 & 9 & 10\\
\hline
\rule[-4mm]{0mm}{9mm} f(n) & 3 & 5 & 7 & 11 & 19 & 29 & 47 & 79 & 127 & 191\\
\hline
\rule[-2mm]{0mm}{7mm}\lfloor2^{3(n-2)/4}\rfloor + 1 & 1 & 2 & 2 & 3 & 5 & 9 & 14 & 23 & 39 & 65
\end{array}$
\smallskip
\end{center}
\caption{The values of $f(n)$ for $n \le 10$.  The bottom row is the size of the smallest topology not obtained by Theorem~\ref{thm:triples}.  That is, the bottom row is $1$ more than the bound $2^{3(n-2)/4}$ obtained in Theorem~\ref{thm:triples}, rounded down to the nearest integer.}\label{table:min bad k for n}
\end{table}

We conclude this introduction by outlining the organization of the paper. In Section \ref{sec:Machinery} we recall basic definitions and describe machinery we will use throughout the proofs. This includes the correspondence between topologies and posets, under which open sets correspond to order ideals. We also develop methods to compute the number of order ideals in a poset. In Section \ref{sec:bounds} we prove the main theorems, giving proofs of logarithmic upper bounds for $m(k)$, and consequently exponential lower bounds for $f(n)$. The proofs are constructive in that we explicitly show how to construct a topology on $n$ points having $k$ open sets for $k \in [2,2^{3(n-2)/4}]$. In Section \ref{sec:specified} we apply the constructions from Section \ref{sec:bounds} to the situation where the minimal neighborhood of each point must have at least $m$ points, and obtain a similar interval of obtainable topology sizes. In Section \ref{sec:better efficiency} we discuss instances where the constructions in Section \ref{sec:bounds} are more efficient, giving topologies on fewer points than the bounds suggest. Finally, in Section \ref{sec:comparison}, we make general observations about the sequences, $\{m(k)\}$ and $\{f(n)\}$, comparing them to other known sequences.

\section{Machinery}\label{sec:Machinery}

In this section we define and discuss some of the basic objects studied in this paper.  Many of these definitions and results are well known, but they are presented again here for the sake of completeness.  

We begin by recalling the definition of a topology.

\begin{defn}
A \emph{topology} on a set $X$ is a collection $\mathcal{T}$ of subsets of $X$, such that $\emptyset, X \in \mathcal{T}$, and $\mathcal{T}$ is closed under arbitrary union and finite intersection.  Elements in $\mathcal{T}$ are called \emph{open sets}. The \emph{size} of a topology is the number of open sets.  In other words, the size of the topology is the cardinality of $\mathcal{T}$.
\end{defn}

The following class of topologies is of special importance in this article.

\begin{defn}
A \emph{$T_0$ topology} on a set $X$ is a topology on $X$ such that, for any pair of distinct points in $X$, there exists an open set containing one of these points and not the other.  In other words, any two points in a $T_0$ topology can be distinguished topologically.
\end{defn}

In this paper we are only concerned with topologies on finite sets $X$. As $X$ has only finitely many subsets, a topology on $X$ is in fact closed under arbitrary intersection.  Consequently, for a point $x \in X$, we can form the minimal open set containing $x$ by taking the intersection  
\begin{equation*}
U_x = \bigcap_{\substack{ U \in \mathcal{T} \\x \in U }} U.
\end{equation*}
\noindent These minimal open sets determine $\mathcal{T}$, since
\begin{equation*}
U = \bigcup_{x \in U} U_x
\end{equation*}
\noindent for all $U \in \mathcal{T}$.

For distinct $x$ and $y$, minimality implies that the sets $U_x$ and $U_y$ are either disjoint, or one is contained in the other. Thus we can make the following definition.

\begin{defn}
For a topology $\mathcal{T}$ on a finite set $X$, let $P(\mathcal{T})$ be the preorder relation on 
$X$ obtained by setting $x \leq y$ when $U_x \subseteq U_y$.
\end{defn}

This assignment is a well-known bijection, as recorded in the following lemma.  For more background, the reader is referred to Alexandroff's work \cite{alexandroff}. 

\begin{lem}\label{lem:Bijection}
For a finite set $X$, the assignment $\mathcal{T} \mapsto P(\mathcal{T})$ gives a bijective correspondence between topologies on $X$ and preorders on $X$. Under this assignment, $T_0$ topologies correspond to partial orders.
\end{lem}

There is a standard way to collapse a topology $\mathcal{T}$ on a set $X$ into a $T_0$ topology of the same size.  First, let $X^0$ be the set of equivalence classes formed by the relation ``$x \sim y$ if $U_x = U_y$'', and let $\pi \colon X \to X^0$ be the canonical projection. One then obtains a $T_0$ topology $\mathcal{T}^0$ on $X^0$ by setting
\begin{equation}\label{eqn:collapse}
\mathcal{T}^0 = \{ \pi(U) \mid U \in \mathcal{T} \}.
\end{equation}
\noindent The size of $\mathcal{T}^0$ is clearly equal to the size of $\mathcal{T}$. Furthermore, $P(\mathcal{T}^0)$ is the poset obtained from the preorder $P(\mathcal{T})$ in the standard way by identifying elements $x$ and $y$ such that $x \leq y$ and $y \leq x$.

\begin{example}\label{ex:poset/top example}
Let $\mathcal{T}$ be the topology on $\{1, \ldots, 8\}$ with minimal open sets $U_1 = \{1\}$, $U_2 = \{2\}$, $U_3 = \{1,2,3\}$, $U_4 = \{1,2,4\}$, $U_5 = \{5\}$, $U_6 = \{1,2,4,5,6\}$, and $U_7 = U_8 = \{1,2,4,5,6,7,8\}$.  This is not a $T_0$ topology because the points $7$ and $8$ are not distinguishable topologically.  The induced $T_0$ topology, $\mathcal{T}^0$, is homeomorphic to the topology on $\{1, \ldots, 7\}$ with minimal open sets $U_1 = \{1\}$, $U_2 = \{2\}$, $U_3 = \{1,2,3\}$, $U_4 = \{1,2,4\}$, $U_5 = \{5\}$, $U_6 = \{1,2,4,5,6\}$, and $U_7 = \{1,2,4,5,6,7\}$.  The poset $P(\mathcal{T}^0)$ is depicted in Figure~\ref{fig:poset/top example}. 
\end{example}

\begin{figure}[htbp]
\epsfig{file=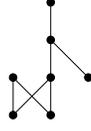,scale=.5}
\caption{The poset $P(\mathcal{T}^0)$ corresponding to the $T_0$ topology $\mathcal{T}^0$ induced by the topology $\mathcal{T}$ in Example~\ref{ex:poset/top example}.} \label{fig:poset/top example}
\end{figure}

The following lemma describes the relationship between the sequences $m(k)$ and $f(n)$, and indicates the role of $T_0$ topologies.

\begin{lem} 
Let $k \ge 2$ be an integer.
\begin{itemize}
 \item[(a)] $m(k)$ is the minimum number such that there exists a $T_0$ topology on $m(k)$ points having $k$ open sets. 
 \item[(b)] If $T(n,k) > 0$ for some $n$, then $T(n',k) > 0$ for all $n' > n$.
\end{itemize}
\end{lem}

\begin{proof}
(a) A topology $\mathcal{T}$ with $k$ open sets on a minimal number of points must be a $T_0$ topology, for otherwise $\mathcal{T}^0$, as defined in equation~\eqref{eqn:collapse}, is a topology with $k$ open sets on fewer points. Thus adding the $T_0$ restriction does not increase the minimal number of points needed for a topology with $k$ open sets.

(b) Suppose $\mathcal{T}$ is a topology of size $k$ on a set $X$ with $n$ points. Pick a point $x \in X$ that is minimal in the preorder $P(\mathcal{T})$, and form the topology $\mathcal{T'}$ by inserting $n'-n$ additional points into $U_x$. Then $\mathcal{T'}$ is a topology of size $k$ on $n'$ points.
\end{proof}

\begin{rem}\label{rem:explainT0}
From the previous lemma, it follows that $f(n)$ is the smallest integer such that $m\big(f(n)\big) > n$.  We stress that the analogous statement is not true for $T_0$ topologies, as there is no analogue of part (b) of the lemma for $T_0$ topologies.  Indeed, a $T_0$ topology on $n$ points necessarily has at least $n+1$ open sets.
\end{rem}

In view of the previous lemma we focus our attention on $T_0$ topologies and posets throughout the rest of the paper. For the remainder of this section we investigate how to calculate the size of a $T_0$ topology using properties of its associated poset.

\begin{defn}
An \emph{order ideal} in a poset $P$ is a subset $I \subseteq P$ such that if $y \in I$ and $x<y$, then $x \in I$.  A \emph{dual order ideal} in $P$ is a subset $I \subseteq P$ such that if $x \in I$ and $x<y$, then $y \in I$.
\end{defn}

Order ideals are sometimes called \emph{down-sets}, while dual order ideals may be called \emph{up-sets} or \emph{filters}.

\begin{defn}
Let $P$ be a poset.  An \emph{antichain} in $P$ is a subset $A \subseteq P$ such that $x$ and $y$ are incomparable for all distinct $x, y \in A$.
\end{defn}

The following lemma is a well-known property of the bijection from Lemma \ref{lem:Bijection}.

\begin{lem}\label{lem:size-ideals-antichains}
Let $\mathcal{T}$ be a $T_0$ topology.  The following correspondences are bijections:
\begin{equation*}
\{\text{open sets in } \mathcal{T}\} \longleftrightarrow \{\text{order ideals in } P(\mathcal{T})\} \longleftrightarrow \{\text{antichains in } P(\mathcal{T})\}.
\end{equation*}
\end{lem}

\begin{defn}
Let $j(P)$ be the number of order ideals in a poset $P$.
\end{defn}

Lemma~\ref{lem:size-ideals-antichains} implies that $j(P(\mathcal{T})) = |\mathcal{T}|$.

\begin{defn}\label{defn:subposets}
For a poset $P$ and an element $x \in P$, let $P_x$ be the poset obtained from $P$ by removing all elements comparable to $x$.  Let $P \setminus x$ be the poset obtained from $P$ by removing only the element $x$.
\end{defn}

Given a poset $P$, the number of order ideals in $P$ can be computed in an iterative manner using the following lemma, which is a key tool in the proof of the main results in the paper.

\begin{lem}\label{lem:iterative counting}
Given a poset $P$ and an element $x \in P$,
\begin{equation*}
j(P) = j(P \setminus x) + j(P_x). 
\end{equation*}
\end{lem}

\begin{proof}
Suppose $x$ is an element of $P$, and consider an antichain $A \subseteq P$.  If $x$ is not in $A$, then $A$ is an antichain in $P \setminus x$.  If $x \in A$, then no element comparable to $x$ is in $A$, so $A \setminus x$ is an antichain in $P_x$.
\end{proof}

Counting the number of antichains in a poset is a \#P-complete problem (see \cite{provan}).  This computational difficulty is the reason that the data presented in Tables~\ref{table:min n for k} and~\ref{table:min bad k for n} do not consider posets with more than $10$ elements.  However, the main proofs in this article build posets by inductively adding a single element at a time, and hence are undisturbed by the computational complexity.

Two elementary operations for constructing posets are the \emph{direct sum} (also called \emph{disjoint union}) and the \emph{ordinal sum} of two posets.

\begin{defn}
Let $P$ and $Q$ be posets on the sets $X$ and $Y$, respectively, with order relations $R$ and $S$, respectively.  The direct sum $P + Q$ is the poset defined on $X \cup Y$, with order relations $R \cup S$.  The ordinal sum $P \oplus Q$ is the poset defined on $X \cup Y$, with order relations $R \cup S \cup \{x \le y \mid x \in X, y \in Y\}$.
\end{defn}

The number of ideals in a poset resulting from these operations can be calculated easily.  The proof of this lemma is straightforward, for example see \cite{ernestege2}. 

\begin{lem}\label{lem:operations}
Let $P$ and $Q$ be posets.  Then
\begin{eqnarray}
j(P+Q) &=& j(P) \cdot j(Q), \text{ and}\label{eqn:j(P+Q)}\\
j(P \oplus Q) &=& j(P) + j(Q) - 1.\label{eqn:j(P oplus Q)}
\end{eqnarray}
\end{lem}

\begin{defn}
Let $\bullet$ denote the poset consisting of a single element.
\end{defn}

The following is an immediate corollary to Lemma~\ref{lem:operations}, which is used in Proposition \ref{prop:singles} to give a simple but efficient algorithm for constructing a topology with $k$ open sets based on the base 2 expansion of $k$.

\begin{cor}\label{cor:x2+1}
For any poset $P$,
\begin{eqnarray*}
j(P + \bullet) &=& 2j(P);\\
j(P\oplus \bullet) &=& j(P) + 1
\end{eqnarray*}
\end{cor}

\begin{proof}
These equalities follow directly from Lemma~\ref{lem:operations} because the poset $\bullet$ has two antichains: the emptyset, and the single element $\bullet$.
\end{proof}

Lemma~\ref{lem:operations} implies the following result, which provides a crude bound on the number of points needed to make a topology with a prescribed number of open sets.

\begin{cor}\label{cor:binary/factor bound}
For all $k \ge 3$, 
\begin{equation*}
m(k) \le \min\left\{1 + m(k-1), \min\limits_{\substack{1<d<k \\ d \mid k}} \{m(d) + m(k/d)\}\right\}.
\end{equation*}
\end{cor}

\section{Exponential bounds}\label{sec:bounds}

Here we describe three related logarithmic upper bounds for $m(k)$.  In turn, these yield exponential lower bounds for $f(n)$, and consequently, exponentially large intervals of $k$ for which $T(n,k) > 0$.  The proofs of Propositions~\ref{prop:singles} and~\ref{prop:doubles} and Theorem~\ref{thm:triples} are constructive: given an integer $k \ge 2$, a poset $P$ having a ``small'' number of elements is built so that $j(P) = k$.  For large values of $n$, Corollaries~\ref{cor:singles}, \ref{cor:doubles}, and~\ref{cor:triples} give successively larger lower bounds for $f(n)$.  It may be possible to increase this bound even further, although just how much further the function $f(n)$ can be increased is still an open question.

One of the key objects in this section is the binary expansion of $k$.

\begin{defn}
Set $\ell = \ell(k) = \lfloor \log_2 k \rfloor$.
\end{defn}

\begin{defn}
Given a positive integer $k = \epsilon_{\ell} 2^{\ell} + \cdots + \epsilon_1 2^1 + \epsilon_0 2^0$ where $\epsilon_i \in \{0,1\}$ and $\epsilon_{\ell} = 1$, let $k_{\textsf 2}$ be the string $\epsilon_{\ell} \cdots \epsilon_1 \epsilon_0$.  Each $\epsilon_i$ is a \emph{bit}, and a bit will henceforth be written in sans-serif font as $\textsf{0}$ or $\textsf{1}$.
\end{defn}

The constructions in Propositions~\ref{prop:singles} and~\ref{prop:doubles} and Theorem~\ref{thm:triples} are similar in that they each give a blueprint for constructing a poset with $k$ elements based on the string $k_{\textsf 2}$, while trying to use as few elements as possible. Theorem~\ref{thm:triples} gives the best bound for $m(k)$ when $k \geq 10$.  However, it is also the most complex of the three procedures.  We include the other methods for three main reasons: in some cases the simpler methods are more effective, the construction in Proposition~\ref{prop:singles} is partly used in the proof of Theorem~\ref{thm:triples}, and the proof of Proposition~\ref{prop:doubles} elucidates the proof of Theorem~\ref{thm:triples} by motivating and explaining the ideas behind the more complicated variant. It should be noted that the construction in Proposition~\ref{prop:singles} has appeared previously, for example, see \cite{vollert}.

In each construction given in this section, we read the string $k_\textsf{2}$ from left to right, building up the poset at each bit. We start with the empty poset at the first bit, and add a disjoint element to the poset for each new bit examined. At times we add maximal elements, covering selected parts of the poset, to adjust for the value of recently read bits. The difference in the constructions lies in how and when the maximal elements are added. A common aspect of each is that the disjoint elements added with the appearance of each bit form a maximal antichain of length $\ell$. This observation is useful for drawing Hasse diagrams: we will draw this antichain at the lowest level, and the elements arising from the values of the bits will be positioned over it.

\begin{prop}\label{prop:singles}
For all $k \ge 2$,
\begin{equation*}
m(k) \le 2\lfloor \log_2 k \rfloor.
\end{equation*}
\end{prop}

\begin{proof}
Let $k \geq 2$ be given and consider the binary expansion $k_{\textsf 2} = \epsilon_{\ell} \cdots \epsilon_1 \epsilon_0$, where $\epsilon_{\ell} = \textsf{1}$. We inductively form posets $P_0, \ldots, P_{\ell}$ with the property that 
\[ \big(j(P_i)\big)_{\textsf{2}} = \epsilon_{\ell} \cdots \epsilon_{\ell - i}\]
for each $i$. In particular, $j(P_\ell) = k$. 

Let $P_0$ be the empty set.  For each $i > 0$, consider the bit $\epsilon_{\ell-i}$, and define
\begin{equation*}
P_i = \begin{cases}
P_{i-1} + \bullet, & \text{if}~ \epsilon_{\ell-i} = \textsf{0};\\
\left(P_{i-1} + \bullet\right) \oplus \bullet, & \text{if}~ \epsilon_{\ell-i}= \textsf{1}.
\end {cases}
\end{equation*}
\noindent Using Corollary \ref{cor:x2+1}, we see that
\begin{equation*}
j(P_i) = \begin{cases}
2j(P_{i-1}), & \text{if}~ \epsilon_{\ell-i} = \textsf{0};\\
2j(P_{i-1}) +1, & \text{if}~ \epsilon_{\ell-i}=\textsf{1}.
\end {cases}
\end{equation*}
\noindent Therefore $j(P_i)$ has binary expansion $\epsilon_{\ell} \epsilon_{\ell - 1}\cdots \epsilon_{\ell-i}$.

The number of elements used in $P_{\ell}$ is $\ell +t - 1$, where $t$ is the number of $\textsf{1}$s in $k_2$. An example of the poset $P_\ell$ for $k = 105$ is drawn in Figure~\ref{fig:singles-ex}.  We have $t \leq \ell + 1$, so $m(k) \le 2\lfloor \log_2 k \rfloor$.
\end{proof}

\begin{cor}\label{cor:singles}
For all $n \ge 1$,
\begin{equation*}
f(n) > 2^{n/2}.
\end{equation*}
\noindent That is, $T(n,k) > 0$ for all $k \in [2, 2^{n/2}]$.
\end{cor}

\begin{figure}[htbp]
\centering
\epsfig{file=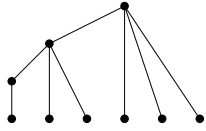,scale=.5}
\caption{The method of Proposition~\ref{prop:singles} applied to $k = 105$, where $k_{\textsf 2} = \textsf{1101001}$.}\label{fig:singles-ex}
\end{figure}

Note that $2\lfloor \log_2 105 \rfloor$ is greater than the number of elements in the poset in Figure~\ref{fig:singles-ex}, but this should not be surprising given the number of $\textsf{0}$s in $105_{\textsf 2}$.  Situations where the procedures of this section may be more efficient will be discussed in Section~\ref{sec:better efficiency}.

The procedure described in the proof of Proposition~\ref{prop:singles} examines one bit of $k_{\textsf 2}$ at a time, adding an element for each position in the string, and possibly adding another element if the bit is $\textsf{1}$, using Corollary \ref{cor:x2+1} to keep track of the number of ideals.  Proposition \ref{prop:doubles} below increases the efficiency by looking at pairs of bits at a time. To do this, we first need an appropriate replacement for Corollary \ref{cor:x2+1} to keep track of the number of ideals.

\begin{defn}
A poset is of \emph{double type} if it contains a dual order ideal isomorphic to the poset $\bullet \oplus \bullet$.
\end{defn}

The importance of the poset $\bullet \oplus \bullet$ is that $j(\bullet \oplus \bullet) = 3$, and it also has a dual order ideal $\bullet$ with $j(\bullet) = 2$. This allows us to adjust for the values of binary substrings $\textsf{11}$ and $\textsf{10}$ in $k_\textsf{2}$ by adding a single maximal element, as the following lemma shows.

\begin{lem}\label{lem:doubletype}
Given a poset $P$ of double type, and $r \in \{2,3\}$, there is a poset $P'$ of double type and with $j(P') = 4 j(P)+r$, formed by adding three elements to the poset $P$. 
\end{lem}

\begin{proof}
Add two elements to $P$ to form the poset $Q = P + \{x_1\} + \{x_2\}$.  By Corollary \ref{cor:x2+1}, we have $j(Q) = 4 j(P)$. 

If $r=2$ (that is, $r_{\textsf 2} = \textsf{10}$), form $P'$ by adding an element $y$ to $Q$, greater than everything except $x_2$. The subposet $\{x_1 \lessdot y\} \cong \bullet \oplus \bullet$ is a dual order ideal in $P'$, and thus $P'$ is of double type.  Applying Lemma~\ref{lem:iterative counting} (with $x = y$) implies that
\begin{equation*}
j(P') = j(Q) + j(\{x_2\}) = 4j(P) + 2.
\end{equation*}

Similarly, if $r=3$ (that is, $r_{\textsf 2} = \textsf{11}$), form $P'$ by adding an element $y$ to $Q$, greater than everything except the dual order ideal $\bullet \oplus \bullet$ required to be in $P$.  This $\bullet \oplus \bullet$ is still a dual order ideal in $P'$, so $P'$ is of double type.  Furthermore, again by Lemma \ref{lem:iterative counting},
\begin{equation*}
j(P') = j(Q) + j(\bullet \oplus \bullet) = 4j(P) + 3.
\end{equation*}

In each case, $P'$ is a poset of double type with $j(P') = 4 j(P)+r$, obtained by adding three elements to $P$.
\end{proof}

\begin{prop}\label{prop:doubles}
For all $k \ge 2$,
\begin{equation*}
m(k) \le (3/2) \lfloor \log_2 k \rfloor + 1.
\end{equation*}
\end{prop}

\begin{proof}
For a given integer $k \ge 2$, we construct a poset $P$ with $k$ open sets. As in the proof of Proposition~\ref{prop:singles}, let $k_{\textsf 2} = \epsilon_{\ell} \cdots \epsilon_1 \epsilon_0$ be the binary expansion of $k$. We will inductively construct posets $P_i$ for certain $i \in [0,\ell]$ with the property that $\big(j(P_i)\big)_{\textsf 2} = \epsilon_\ell \cdots \epsilon_{\ell-i}$. The process ends when $P_\ell$ is defined, and we take $P := P_\ell$.

If $\epsilon_\ell$ is the only bit equal to $\textsf{1}$, then set $P$ to be a poset consisting of $\ell$ disjoint elements. Otherwise, let $s$ be the smallest positive integer such that $\epsilon_{\ell-s} = \textsf{1}$. Let $P_s$ be the poset $(s \cdot \bullet) \oplus \bullet$, where $s \cdot Q$ denotes the direct sum $Q + \cdots + Q$ of $s$ copies of the poset $Q$.   Then $j(P_s) = 2^s+1$, which has binary expansion $\epsilon_\ell \cdots \epsilon_{\ell-s}$. Furthermore, the poset $P_s$ has $s+1$ elements and is of double type.

The remainder of the proof is inductive.  Assume that $P_i$ has been defined, is of double type, and that $j(P_i)$ has binary expansion $\epsilon_\ell \cdots \epsilon_{\ell-i}$. Consider the bit $\epsilon_{\ell - (i+1)}$. If $\epsilon_{\ell-(i+1)} = \textsf{0}$, then set $P_{i+1} = P_i + \bullet$. Otherwise, unless $\ell - (i+1) = 0$, the substring $\epsilon_{\ell - (i+1)}\epsilon_{\ell-(i+2)}$ is either $\textsf{11}$ or $\textsf{10}$. By Lemma \ref{lem:doubletype}, we can form a poset $P_{i+2}$ of double type such that $j(P_{i+2})$ has binary expansion $\epsilon_\ell \cdots \epsilon_{\ell-(i+2)}$, by adding three elements to $P_i$. If $\epsilon_{\ell - (i+1)} = \textsf{1}$ and $\ell - (i+1) = 0$, then set $P_\ell = (P_i + \bullet) \oplus \bullet$.

An example of the poset as constructed by this procedure for $k = 5550$ is depicted in Figure~\ref{fig:doubles-ex}.

To construct $P$, we first used $s+1$ elements to construct $P_s$, accounting for the leftmost $s+1$ bits in $k_{\textsf 2}$. After that, we either add one element to advance one bit, or add three elements to advance two bits, until the end where two elements may need to be added for the last bit.  Therefore
\begin{eqnarray*}
|P| &\le& (s + 1) + (\ell - s) + \lceil (\ell - s)/2 \rceil\\
    &=& \ell + 1 + \lceil (\ell - s)/2 \rceil \\ 
    &\le& \ell + 1 + \lceil (\ell - 1)/2 \rceil.
\end{eqnarray*}
\noindent Considering cases for the parity of $\ell - 1$, one sees that
\begin{equation*}
 \ell + 1 + \lceil (\ell - 1)/2 \rceil \le (3/2) \lfloor \log_2 k \rfloor + 1,
\end{equation*}
\noindent finishing the proof.
\end{proof}

\begin{cor}\label{cor:doubles}
For all $n \ge 1$,
\begin{equation*}
f(n) > 2^{2(n-1)/3}.
\end{equation*}
\noindent That is, $T(n,k) > 0$ for all $k \in [2,2^{2(n-1)/3}]$.
\end{cor}

\begin{figure}[htbp]
\centering
\epsfig{file=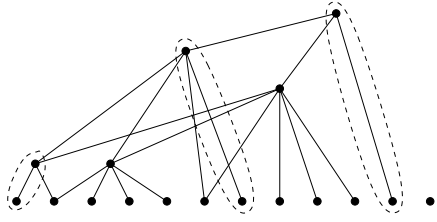,scale=.5}
\caption{The method of Proposition~\ref{prop:doubles} applied to $k = 5550$, where $k_{\textsf 2} = \textsf{1010110101110}$.  The dual order ideals isomorphic to $\bullet \oplus \bullet$ which are defined by the procedure are circled.}\label{fig:doubles-ex}
\end{figure}

Note that $1.5 \lfloor \log_2 5550 \rfloor + 1$ is greater than the number elements in the poset in Figure~\ref{fig:doubles-ex}, but, again, this should not be surprising given the number of $\textsf{0}$s in $5550_{\textsf 2}$.

The bound obtained in Proposition~\ref{prop:doubles} by considering pairs of consecutive bits in $k_{\textsf 2}$ is better than the function obtained in Proposition~\ref{prop:singles}.  In fact this bound can be improved still further by considering triples of consecutive bits in $k_{\textsf 2}$, as shown below, although this is significantly more complicated than the previous methods.  As discussed at the end of the section, there is no analogous method for considering quadruples of consecutive bits in $k_{\textsf 2}$.

\begin{defn}\label{def:tripletype}
A poset is of \emph{triple type} if it contains a dual order ideal isomorphic to one of the following posets, named as indicated.
\begin{equation*} 
\text{Type 1: \ }  \parbox{.75cm}{\epsfig{file=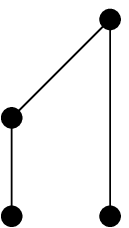, scale=.5}} \hspace{.5in} \text{Type 2: \ } \parbox{1.2cm}{\epsfig{file=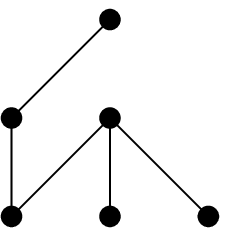, scale=.5}} \hspace{.5in} \text{Type 3: \ } \parbox{1.7cm}{\epsfig{file=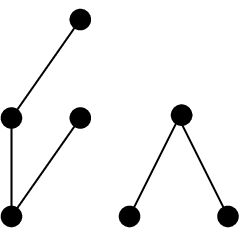, scale=.5}}
\end{equation*} 
\end{defn}

The motivation for this definition is similar to that for double type.  If $P$ is isomorphic to a poset of Type 1, 2, or 3, and $Q$ is the poset obtained by adding three disjoint points to $P$, then for each $r \in \{4,5,6,7\}$, there is a dual order ideal $I$ in $Q$ with $j(I) = r$.

\begin{lem}\label{lem:tripletype}
Given a poset $P$ of triple type, and an integer $r \in \{4,5,6,7 \} $, there is a poset $P'$ of triple type and with $j(P') = 8 j(P)+r$, formed by adding four elements to the poset $P$. 
\end{lem}
\begin{proof}
Add three elements to $P$ to form the poset $Q = P + \{x_1\} + \{x_2\} + \{x_3\}$.  By Corollary \ref{cor:x2+1}, we have $j(Q) = 8 j(P)$. 

Let $I$ be a dual order ideal in $P$ that is isomorphic to one of the posets illustrated in Definition \ref{def:tripletype}, and let $J$ be the dual order ideal $I + \{x_1\} + \{x_2\} + \{x_3\}$ in $Q$. To complete the proof, we will form a new poset $P'$ by adding a maximal element $y$ to $Q$ such that the following three conditions are satisfied
\begin{itemize}
 \item $P'$ is of triple type,
 \item $y > x$ for all $x \in Q \setminus J$,
 \item $j(J_y) = r$, with notation as in Definition~\ref{defn:subposets}.
\end{itemize}
Combined the second and third condition imply that $P'_y = J_y$, and by Lemma \ref{lem:iterative counting} we have
\begin{equation*}
j(P') = j(P' \setminus y) + j(P'_y) = j(Q) + j(J_y) = 8 j(P) + r,
\end{equation*}
\noindent as desired.

There are twelve cases to consider for adding the element $y$, depending on the type of $I$ and the value of $r$. The figures below show how to place $y$ in relation to the dual order ideal $J$ in each case. As $y > x$ for all $x \in Q \setminus J$, this shows how to add $y$ to $Q$. In each case the dual order ideal $I$ is drawn with solid lines, and the dual order ideal making $P'$ of triple type is circled.  Of the four new elements in each figure, the maximal of these is $y$.  The first figure in each row corresponds to the case $r = j(J_y) = 4$, the second to the case $r = j(J_y) = 5$, the third to $r = j(J_y) = 6$, and the fourth to $r = j(J_y) = 7$.

\begin{equation*}\tag{Type 1}
\epsfig{file=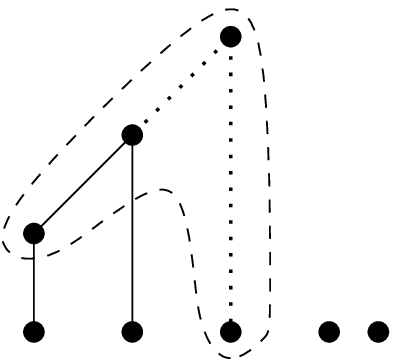,scale=.5} \hspace{.35in} \epsfig{file=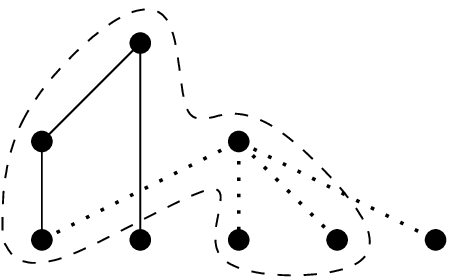,scale=.5} \hspace{.35in}
\epsfig{file=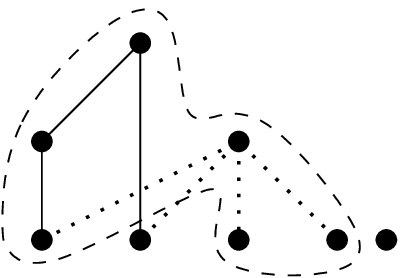,scale=.5} \hspace{.35in} \epsfig{file=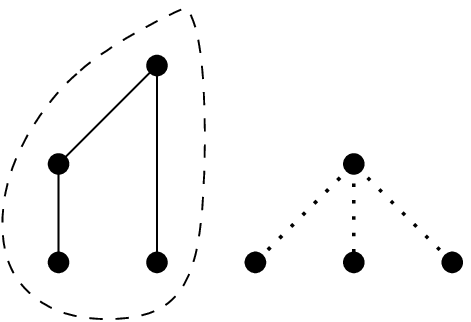,scale=.5}
\end{equation*}
\begin{equation*}\tag{Type 2}
\epsfig{file=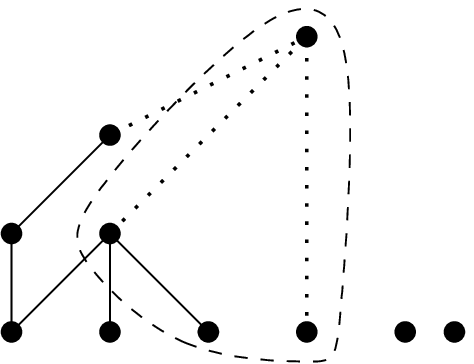,scale=.5} \hspace{.25in} \epsfig{file=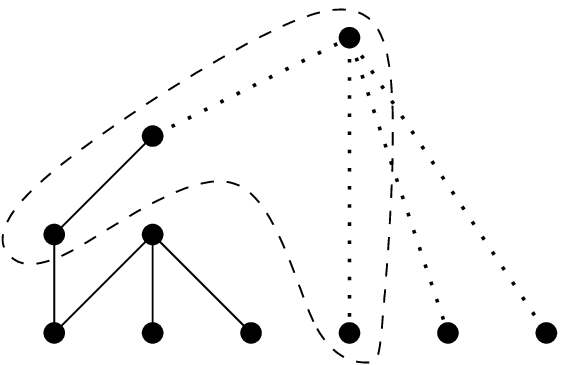,scale=.5} \hspace{.25in}
\epsfig{file=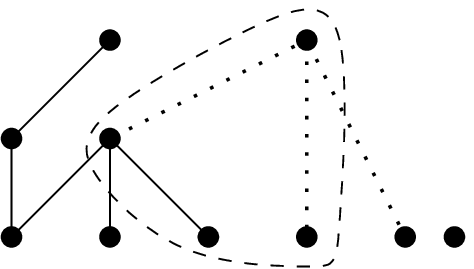,scale=.5} \hspace{.25in} \epsfig{file=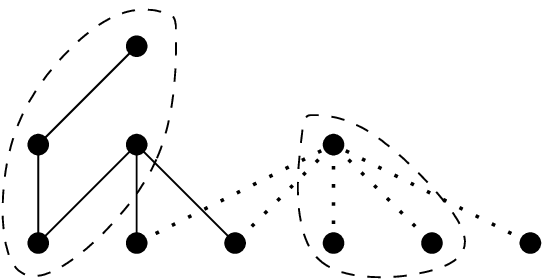,scale=.5}
\end{equation*}
\begin{equation*}\tag{Type 3}
\epsfig{file=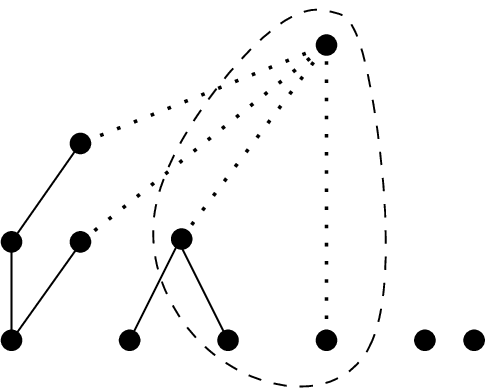,scale=.5} \hspace{.25in} \epsfig{file=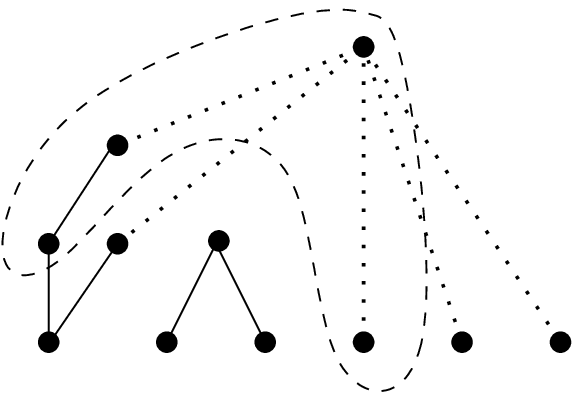,scale=.5} \hspace{.25in}
\epsfig{file=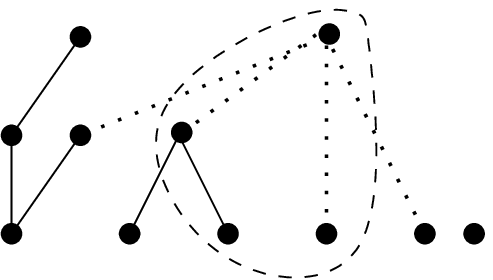,scale=.5} \hspace{.25in} \epsfig{file=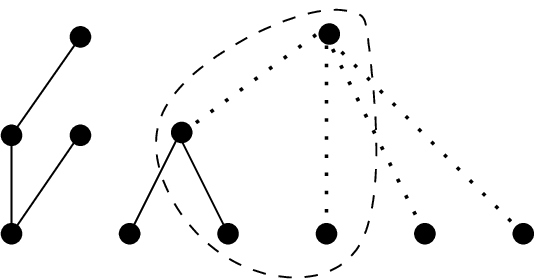,scale=.5}
\end{equation*}
\end{proof}

\begin{thm}\label{thm:triples}
For all $k \ge 2$,
\begin{equation*}
m(k) \le (4/3) \lfloor \log_2 k \rfloor + 2.
\end{equation*}
\end{thm}

\begin{proof}
The approach is similar to the proof of Proposition \ref{prop:doubles}, and we only outline the construction.  Let $k \geq 2$ be a fixed integer, and let $k_{\textsf 2} = \epsilon_\ell\cdots\epsilon_0$ be the binary expansion of $k$. We construct a poset $P$ with $j(P)=k$.  If $k$ has fewer than three bits equal to $\textsf{1}$ in its binary expansion, use the construction in Proposition \ref{prop:singles} to obtain a poset with no more than $\ell +1$ elements. Otherwise, let $s$ be such that $\epsilon_{\ell-s}$ is the third nonzero bit from the left in $k_\textsf{2}$. Using the construction in Proposition \ref{prop:singles} we obtain a poset $P_s$ with $s +2$ elements such that $j(P_s)$ has binary expansion $\epsilon_\ell \cdots \epsilon_{\ell-s}$. Observe that $P_s$ is of triple type as it contains a dual order ideal isomorphic to Type 1 in Definition \ref{def:tripletype}.

As in the proof of Proposition \ref{prop:doubles}, we now move rightward in the binary expansion of $k$. If we encounter the bit $\textsf{0}$, we add a single disjoint point to our poset and move on. If we encounter the bit $\textsf{1}$, we consider this bit and the two immediately following it. They form one of the subsequences $\textsf{100}, \textsf{101}, \textsf{110}$ or $\textsf{111}$.  In each case the corresponding integer belongs to the set $\{4,5,6,7\}$, and we can apply Lemma \ref{lem:tripletype} to obtain a new poset of triple type incorporating the three bits under scrutiny, by adding four elements. Finally, when there are $i < 3$ bits left we can incorporate them into the poset by adding $i + 1$ points, using Corollary \ref{cor:x2+1} if $i=1$ and Lemma \ref{lem:doubletype} if $i=2$.

A counting argument similar to the one in Proposition \ref{prop:doubles} shows that
\begin{eqnarray*}
|P| & \leq & (s+2) + (\ell - s) + \lceil (\ell - s)/3 \rceil \\ 
    & = & \ell + 2 + \lceil (\ell - s)/3 \rceil \\
    & \leq & \ell + 2 + \lceil (\ell -2)/3 \rceil.
\end{eqnarray*}
\noindent Examination of cases based on the remainder $\ell$ mod 3 gives that 
\begin{equation*}
  \ell + 2 + \lceil (\ell -2)/3 \rceil \le (4/3) \lfloor \log_2 k \rfloor + 2,
\end{equation*}
\noindent finishing the proof.
\end{proof}

\begin{cor}\label{cor:triples}
For all $n \ge 1$,
\begin{equation*}
f(n) > 2^{3(n-2)/4}.
\end{equation*}
\noindent That is, $T(n,k) > 0$ for all $k \in [2,2^{3(n-2)/4}]$.
\end{cor}

The successive results in Propositions~\ref{prop:singles} and~\ref{prop:doubles} and Theorem~\ref{thm:triples} suggest that even better bounds might be obtained by adapting the constructions to consider four bits of $k_{\textsf 2}$ at a time, for any $k \ge 2$.  However, our current approach does not translate directly into an approach for quadruples of digits.  More precisely, we cannot add four disjoint points to the poset, and a single maximal element, and maintain the existence of a collection of dual order ideals having $8$, $9$, $10$, $11$, $12$, $13$, $14$, and $15$ order ideals, respectively.  We do not rule out the possibility that another technique might be employed to improve the result of Theorem~\ref{thm:triples}, but leave that as a question for future research.

\section{Specified minimal set sizes} \label{sec:specified}

The results in the previous section can be generalized by looking at topologies where the minimal open sets $\{U_x\}$ have specified sizes.  An extremal case of this, related to cardinalities of distributive lattices with a specified number of join-irreducibles of each rank, is treated in \cite{stanley-extremal}.  Additionally, unlabeled distributive lattices with fewer than $50$ elements and an arbitrary given number of irreducible elements are studied in \cite{erneheitzigreinhold}.  One version of this generalization is very easy to handle by modifying the construction described in Theorem~\ref{thm:triples} to produce topologies with specified minimal set sizes.

\begin{defn}
Let $T_m(n,k)$ be the number of topologies on $n$ points having $k$ open sets, where the smallest neighborhood of each point has at least $m$ elements.
\end{defn}

\begin{prop}
$T_m(n,k) > 0$ for all $n \ge m$, $m \ge 1$, and $k \in [2,2^{\frac{3(n-2)}{3m+1}}].$
\end{prop}

\begin{proof}
In a topology $\mathcal{T}$, the smallest neighborhood of a point $x$ is the set $U_x$. The sets $U_x$ with fewest elements are those where $x$ is minimal in the preorder  $P(\mathcal{T})$.  

In the procedure described in the proof of Theorem~\ref{thm:triples}, the minimal elements of the poset form an antichain of size $\ell$, corresponding to each bit after $\epsilon_{\ell}$ in the step-by-step reading of $k_{\textsf 2}$.  Therefore, requiring the smallest neighborhood of each point in $\mathcal{T}$ to contain at least $m$ points simply means replacing each of these $\ell$ elements by a set of cardinality at least $m$.  Thus, to make such a topology with $k$ open sets, a similar argument to that in the proof of the theorem shows that one needs at most $m\ell + 2 + \lceil (\ell - 2)/3 \rceil$ elements.  As in the proof of the theorem,
\[ m\ell + 2 + \lceil (\ell - 2)/3 \rceil \le (m + 1/3) \ell + 2,\]
and the result follows.
\end{proof}

\section{Better efficiency}\label{sec:better efficiency}

The main result of this paper, Theorem~\ref{thm:triples}, gives a procedure to construct a topology having $k$ open sets, needing one extra point in the topology for each triple of bits after the first three $\textsf{1}$s in the binary expansion $k_{\textsf 2}$.  There may be some situations where this procedure requires fewer points than the bounds suggest, and we highlight a few of these here.  

First of all, if the binary expression $k_{\textsf 2}$ includes many $\textsf{0}$s, then there may be large portions of this expression that get skipped over by the procedure, and thus fewer triples contribute an element to the poset.

Another way to increase the efficiency of this type of procedure would be to note patterns of consecutive digits in the string $k_{\textsf 2}$.  For example, suppose that $k = 2^{2^r} - 1$.  Thus $\ell = 2^r - 1$ and the binary string $k_{\textsf 2}$ consists of $2^r $ repeated  $\textsf{1}$s.  Then one can parse the string $k_{\textsf 2}$ as
\begin{equation*}
\textsf{1} \mid \textsf{1} \mid \textsf{11} \mid \textsf{1111} \mid \textsf{11111111} \mid \ldots,
\end{equation*}
\noindent where each section is identical to the union of all sections to the left.  Thus a new section with $2^s$ $\textsf{1}$s can be handled by finding a dual order ideal in the poset with $2^{2^s}-1$ antichains, similarly to the procedure in the proof of Theorem~\ref{thm:triples}.  An example of this for $2^{2^4} - 1$ is depicted in Figure~\ref{fig:65535}.

\begin{figure}[htbp]
\epsfig{file=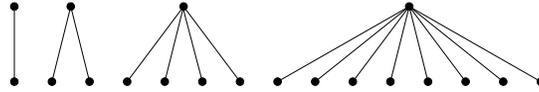,scale=.5}
\caption{An efficient way to draw a poset with $65535$ antichains, using 19 elements.}\label{fig:65535}
\end{figure}

As suggested by Figure~\ref{fig:65535} and Lemma~\ref{lem:operations}, if the number of open sets desired factors conveniently well, this may also reduce the number of points needed in the topology.

Fix positive integers $a$ and $b$.  If the desired number of open sets is
\begin{equation*}
k = 1 + 2^a + 2^{2a} + \cdots + 2^{ba},
\end{equation*}
\noindent then the procedure in Proposition~\ref{prop:singles} gives a poset having $(a+1)b$ elements and $k$ antichains.  Figure~\ref{fig:powers of 2} depicts such a poset when $a=3$ and $b=4$ (that is, $k = 4681$).

\begin{figure}[htbp]
\epsfig{file=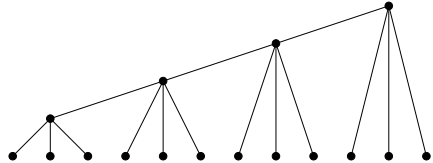,scale=.5}

\caption{The procedure in Proposition~\ref{prop:singles} applied to $k = 4681$.}\label{fig:powers of 2}
\end{figure}

Now consider an integer of the form $k = x(1 + 2^a + 2^{2a} + \cdots + 2^{ba})$, where $\ell (x) + 1 \le a$. The binary expansion of $k$ consists of $b+1$ repeated instances of the binary expansion of $x$.  In this situation, due to Lemma~\ref{lem:operations}, there exists a poset having $k$ antichains and at most
\begin{equation*}
(a+1)b + (4/3)\lfloor \log_2 x \rfloor + 2
\end{equation*}
\noindent elements. Thus, integers $k$ with repeated patterns in their binary expansion can be handled very efficiently.

\section{Comparison to other sequences} \label{sec:comparison}

Using Stembridge's MAPLE program \cite{posets}, we have calculated the initial values of the sequence $\{m(k)\}$, and these have been entered into \cite{oeis} as entry A137813.  The terms of $m(k)$ are very similar to sequence A003313 of \cite{oeis}, giving the length of a shortest addition chain for an integer, and the constructions in the previous section are in fact similar to those for producing short addition chains and star chains \cite{knuth}.

\begin{defn}
An \emph{addition chain} for $k$ is sequence of integers $x_0, x_1, \cdots, x_n$ such that $x_0 = 1$, $x_n = k$, and each term in the sequence is the sum of two (not necessarily distinct) numbers appearing earlier in the sequence. The \emph{length} of the addition chain $x_0, x_1, \cdots, x_n$ is $n$.
\end{defn}
 
For more information, both historical and mathematical, about addition chains, see \cite{knuth}.
Sequence A003313 of \cite{oeis} is defined as follows. 

\begin{defn}
For a positive integer $k$, let $a(k)$ be the length of the shortest possible addition chain for $k$.
\end{defn}

Interestingly, the sequences $a(k)$ and $m(k)$ agree in their first 100 terms, except for $k = 71$, where $m(71) = 8$, while $a(71) = 9$. It is tempting to wonder whether $a(k)$ is an upper bound for $m(k)$.  Examples suggest that ``short'' addition chains can be realized by posets, but this does not seem to be true for ``long'' addition chains. The division between ``short'' and ``long'' chains is unclear, but seems to lie above the range of values for which it is currently feasible to calculate $m(k)$. The relationship between these sequences is intriguing, and has previously been studied by Vollert in \cite{vollert}.

A concrete relationship between the sequences $m(k)$ and $a(k)$ is a common upper bound.

\begin{defn}
For a positive integer $k$, let $b(k)$ be the length of the shortest possible addition chain for $k$ obtained by using only the methods of factoring and binary expansion.
\end{defn}

This is sequence A117498 in \cite{oeis}. By definition, $b(k)$ is an upper bound for $a(k)$. Also from the definition it follows that $b(k)$ satisfies the inductive equation
\begin{equation*}
b(k) = \begin{cases}
1, & \text{if}~ k = 2;\\
\min\left\{1 + b(k-1), \min\limits_{\substack{1 < d < k \\ d \mid k}} \{b(d) + b(k/d)\}\right\}, & \text{if}~ k > 2.
\end{cases}
\end{equation*}
It follows from Corollary~\ref{cor:binary/factor bound} that $b(k)$ is an upper bound for $m(k)$. The first term where the sequences differ is $k = 23$, where $b(23) = 7$, while $a(23) = m(23) = 6$.

The sequence $f(n)$ has also been entered into \cite{oeis}, as sequence number A137814. 
The initial terms of this sequence are 3, 5, 7, 11, 19, 29, 47, 79, 127, and 191. That these are all prime numbers is not surprising: for a composite number $k$, Lemma \ref{lem:operations} implies that one can efficiently construct a poset with $k$ order ideals as a direct sum of two posets.  Given the relationship between $m(k)$ and $a(k)$ above, it is expected that $\{f(n)\}$ be similar to sequence A003064 of \cite{oeis}, giving the smallest number with addition chains of length $n$.

\section{Acknowledgments}

We are grateful to two thoughtful referees for their careful readings of this manuscript, and for bringing to our attention the extensive work of the Hannover group.

\end{document}